\documentclass[12pt, reqno]{article}
\usepackage{amsmath}
\usepackage{amssymb}
\usepackage{latexsym,bm}
\usepackage{amsthm}
\usepackage{graphicx}

\usepackage[top=2.4cm,bottom=2cm,left=2.4cm,right=2.4cm]{geometry}
\usepackage{algorithm}
\usepackage{algorithmic}
\usepackage{cite}
\usepackage[colorlinks,urlcolor=blue]{hyperref}
\allowdisplaybreaks
\usepackage{dsfont}
\usepackage{multirow}
\usepackage{threeparttable}
\usepackage{slashbox}

\newtheorem{thm}{Theorem}[section]

\newtheorem{lem}{Lemma}[section]

\theoremstyle{definition}
\newtheorem{defn}{Definition}[section]
\theoremstyle{remark}
\newtheorem{rmk}{Remark}[section]
\numberwithin{equation}{section}
\theoremstyle{example}
\newtheorem{exa}{Example}[section]
\numberwithin{equation}{section} \numberwithin{algorithm}{section}
\begin{document}

\bigskip
\bigskip

\bigskip

\begin{center}

\textbf{\large A splitting primal-dual proximity algorithm for
solving composite optimization problems}

\end{center}

\begin{center}
Yu-Chao Tang$^{1}$\footnote{Corresponding author, Email:
hhaaoo1331@aliyun.com; yctang.09@stu.xjtu.edu.cn}, Chuan-Xi
Zhu$^{1}$, Meng Wen$^{2,3}$, Ji-Gen Peng$^{2,3}$,
\end{center}

\begin{center}
1. Department of Mathematics, Nanchang University, Nanchang 330031,
P.R. China\\
2. School of Mathematics and Statistics, Xi'an Jiaotong University,
Xi'an 710049, P.R. China \\
3. Beijing Center for Mathematics and Information Interdisciplinary
Sciences, Beijing, P.R. China\\

\end{center}

\bigskip

\begin{center}
{\bf Abstract}
\end{center}

Our work considers the optimization of the sum of a non-smooth
convex function and a finite family of composite convex functions,
each one of which is composed of a convex function and a bounded linear
operator. This type of problem is associated with many interesting challenges
encountered in the image restoration and image reconstruction fields. We
developed a splitting primal-dual proximity algorithm to solve this
problem. Further, we propose a preconditioned method, of which the
iterative parameters are obtained without the need to know some
particular operator norm in advance. Theoretical convergence
theorems are presented.  We then apply the proposed methods to solve
a total variation regularization model, in which the $L2$ data error
function is added to the $L1$ data error function. The main advantageous
feature of this model is its capability to combine different loss functions.
The numerical results obtained for computed tomography (CT) image reconstruction
demonstrated the ability of the proposed algorithm to reconstruct an image with few
and sparse projection views while maintaining the image quality.

\bigskip
\noindent \textbf{Keywords:} Sparse optimization; Proximity
operator; Saddle-point problem; CT image reconstruction.

 \noindent \textbf{2010 Mathematics Subject Classification:} 90C25;
 65K10.

\section{Introduction}

In this paper, we consider solving the following convex optimization
problem:
\begin{equation}\label{problem1}
\min_{x\in X}\  \sum_{i=1}^{l}F_i (K_i x) + G(x),
\end{equation}
where $l$ is an integer, $X$ and $\{Y_i\}_{i=1}^{l}$ are Hilbert
spaces, the functions $\{F_i\}_{i=1}^{l}$ and $G$ belong in
$\Gamma_{0}(Y_i)$ and $\Gamma_{0}(X)$, respectively, and $K_i :
X\rightarrow Y_i$ is a continuous linear operator for $i=1, 2,
\cdots, l$. Here and in what follows, for a real Hilbert space $H$,
$\Gamma_{0}(H)$ denotes the collection of all proper lower
semi-continuous (LSC) convex functions from $H$ to $(-\infty, +
\infty]$. Based on the assumptions of problem (\ref{problem1}),
the functions $(F_i \cdot K_i)_{1\leq i \leq l}$ may be used to model the data
fidelity term, including smooth and non-smooth measures, and $G$ could be
the indicator function of a convex set or $\ell_{1}$-norm, for example.
Therefore, the optimization model (\ref{problem1}) would be able to accommodate
a combination of different data error functions.

In particular, if $l=1$, then problem (\ref{problem1}) is reduced
to the following
\begin{equation}\label{problem2}
\min_{x\in X}\ F(Kx) + G(x),
\end{equation}
where $F\in \Gamma_{0}(Y)$, $G\in \Gamma_{0}(X)$, and $K:
X\rightarrow Y$ is a continuous linear operator. Under the
assumption that the proximity operator of $F^{*}$ and $G$ are easy
to compute (i.e., it either has a closed-form solution or can be efficiently
computed with high precision), Chambolle and
Pock\cite{chambolleandpock2011} proposed a primal-dual proximity
algorithm to solve problem (\ref{problem2}). They proved the
convergence of the proposed iterative algorithm in finite
dimensional Hilbert spaces. They also pointed out the relationship
between the primal-dual proximity algorithm and other existing
algorithms, such as extrapolational gradient
methods\cite{popov1980}, the Douglas-Rachford splitting
algorithm\cite{combettes2007}, and the alternating direction method of
multipliers\cite{boyd1}. Further, in \cite{pock1}, they introduced a
precondition technique to compute the step size of the algorithm
automatically. Numerical experiments showed that the preconditioned
primal-dual proximity algorithm outperforms the primal-dual proximity
algorithm in \cite{chambolleandpock2011}. He and Yuan\cite{he2012}
studied the convergence of the primal-dual proximity algorithm by
presenting this algorithm of Chambolle and
Pock\cite{chambolleandpock2011} in the form of a proximal point
algorithm in infinite dimensional Hilbert spaces.
Condat\cite{condat2013} also obtained the convergence of the
primal-dual proximity algorithm but from a different point of view,
namely by studying the following optimization problem:
\begin{equation}\label{problem3}
\min_{x}\ P(x) + G(x) + F(Kx),
\end{equation}
where $P: X\rightarrow R$ is convex, differentiable, and its
gradient is Lipschitz continuous, and $G$, $F$, and $K$ are the same
as in problem (\ref{problem2}). If $P(x) = 0$, then problem
(\ref{problem3}) reduces to problem (\ref{problem2}).
Condat\cite{condat2013} proposed an efficient iterative algorithm
for solving (\ref{problem3}) and also proved its convergence based
on Krasnoselskii-Mann iteration methods. The primal-dual proximity
algorithm is a special case of Condat's algorithm by setting
$P(x)=0$. Further, Condat proved the convergence of the primal-dual
proximity algorithm in finite dimensional spaces where the
parameters were relaxed from $\sigma \tau \|K\|^2 < 1$ to $\sigma
\tau \|K\|^{2} \leq 1$. These are very useful results because it
becomes possible to fix one parameter in the algorithm, allowing the
other parameter to be tuned in practice.

If we  let $F_0(x) = G(x), K_0 = I$, then the problem
(\ref{problem1}) can also be formulated as follows,
\begin{equation}\label{problem4}
\min_{x}\ \sum_{i=0}^{l}F_{i}(K_i x).
\end{equation}
Setzer et al.\cite{setzer2010} proposed to use an alternating split
Bregman method\cite{goldstein2009} to solve the problem
(\ref{problem4}) and proved\cite{setzer2009, setzer2011}
that this method coincided with the
alternating direction method of multipliers, which can be
interpreted as a Douglas-Rachford splitting algorithm applied to the
dual problem. However, this iterative algorithm always incorporates
linear equations, which are required to be solved either explicitly or
approximately. Condat\cite{condat2013} considered the following
general composite optimization problem,
\begin{equation}\label{problem5}
\min_{x}\ \sum_{i=1}^{l}F_i (K_i x) + G(x) + Q(x),
\end{equation}
where the linear operators $\{K_i\}_{i=1}^{l}$, the functions
$\{F_i\}_{i=1}^{l}$, and $G$ are the same as in problem
(\ref{problem1}), apart from the fact that the function $Q(x)$
is also convex, differentiable, and displays a Lipschitz continuous
gradient. He obtained an iterative algorithm to solve problem (\ref{problem5})
by recasting it as problem (\ref{problem1}) using the product spaces
method. The iterative parameters in the algorithm introduced in
\cite{condat2013} rely on the estimation of the operator norm
$\|\sum_{i=1}^{l}K_{i}^{*}K_i\|$, which may affect its practical use.
To overcome this disadvantage, we propose a preconditioned
iterative algorithm to solve problem (\ref{problem1}), where the
iterative parameters are calculated self-adaptively. If the function
$Q(x)$ is equal to the least-squares loss function, i.e., $Q(x) =
\frac{1}{2}\|Ax-b\|_{2}^{2}$, then problem (\ref{problem5})
could be viewed as a special case of problem (\ref{problem1}).

The primal-dual algorithm is a very flexible method to
solve the optimization problem (\ref{problem2}), which has wide
potential application in image restoration and image reconstruction, for
example,\cite{zhang2011,combettes2012,chenpj-nanjing2013, chenpj2013,
tangyuchao20121,tangyuchao2013, tangyuchao2015}. Sidky et
al.\cite{sidky2} applied the primal-dual proximity algorithm
introduced by Chambolle and Pock \cite{chambolleandpock2011, pock1}
to solve various convex optimization problems. For example,
\begin{align}
& \min_{x}\ \frac{1}{2}\|Ax-b\|_{2}^{2};\label{l2} \\
 & \min_{x}\
\frac{1}{2}\|Ax-b\|_{2}^{2}, \quad s.t., x\geq 0;\label{constrainedl2}\\
& \min_{x}\
\frac{1}{2}\|Ax-b\|_{2}^{2} + \lambda \|x\|_{TV}; \label{L2-TV}\\
& \min_{x}\ \|Ax-b\|_{1} + \lambda \|x\|_{TV}; \label{L1+TV} \\
& \min_{x}\ KL(Ax,b) + \lambda \|x\|_{TV} \label{KL+TV},
\end{align}
where $\|\cdot\|_{1}$ represents the $\ell_{1}$-norm,
$\|\cdot\|_{2}$ represents the $\ell_{2}$-norm, $\|\cdot\|_{TV}$
denotes the total variation semi-norm, $KL(\cdot, \cdot)$ denotes
the Kullback-Leibler (KL) divergence, and $\lambda
>0$ is the regularization parameter balancing the data error term
and the regularization term. The least-squares data error term is
used widely in computed tomography (CT) image reconstruction.
It is modeled by adding Gaussian noise to the collected data and
the $L1$ data loss function has the advantage
of reducing the impact of image sampling with large outliers. They
studied the application of this convex optimization problem in
CT image reconstruction to demonstrate the
performance of these different models under appropriate levels of noise.
The numerical results showed that the primal-dual proximity algorithm
can efficiently solve these problems and it exhibited very good
performance in terms of reconstructing simulated breast CT data.
The work of Sidky et al.\cite{sidky2} motivated us to introduce a
general composite optimization problem for image reconstruction.
Then, the above optimization problem (\ref{L2-TV}) and (\ref{L1+TV})
would be a special case of our proposed optimization problem.

The purpose of this paper is to introduce a splitting primal-dual
proximity algorithm for solving problem (\ref{problem1}) and to
propose a preconditioning technique to improve the performance of this
algorithm. In addition, theoretical convergence theorems
are also provided. We then demonstrate the performance of our proposed algorithms
by applying them to solve a composite optimization problem, which has wide application
in the image restoration and image reconstruction fields.

The rest of this paper is organized as follows. In Section
\ref{section2}, we provide selected background information on
convex analysis. In Section \ref{section3}, we briefly review the
primal-dual proximal algorithm, together with one of its preconditioned
techniques. These iterative algorithms are employed to develop a
splitting primal-dual proximal algorithm for solving problem (\ref{problem1})
and the results are presented in Section \ref{section4}. In
Section \ref{section5}, we apply the proposed iterative algorithm to solve
a particular convex optimization model, which is relevant to the CT image
reconstruction problem. We use numerical results to illustrate the capabilities
of our proposed algorithm in Section \ref{section6}.
Finally, we offer some conclusions.

%We illustrate the flexibility of the proposed splitting on CT image
%reconstruction problem.

\section{Preliminaries}\label{section2}

In this section, we introduce some definitions and notations. Let
$H$ be a real Hilbert space, with its inner product $\langle \cdot,
\cdot \rangle$ and norm $\| \cdot \| = \langle \cdot, \cdot
\rangle^{1/2}$. We denote by $\Gamma_{0}(H)$ the set of proper
lower semicontinuous (LSC), convex functions from $H$ to $(-\infty,
+\infty]$.

\begin{defn}
Let $f$ be a real-valued convex function on $H$, for which the proximity
operator $prox_{f}$ is defined by
\begin{equation}
\begin{aligned}
prox_{f} : H & \rightarrow H  \\
x & \mapsto \arg\min_{y\in H}\ f(y) + \frac{1}{2}\|x-y\|_{2}^{2}.
\end{aligned}
\end{equation}
\end{defn}
Let $C$ be a nonempty closed convex set of $H$. The indicator
function of $C$ is defined on $H$ as
\begin{equation}
\iota_{C}(x) = \left \{
\begin{array}{ll} 0, & \textrm{ if } x \in C, \\
+\infty, & \textrm{otherwise}
\end{array} \right.
\end{equation}
It is easy to see that the proximity operator of the indicator
function is the projection operator onto $C$. That is,
$prox_{\iota_{C}}(x)=P_{C}(x)$, where $P_{C}$ represents the
projection operator onto $C$.

For some simple functions, there is a closed-form solution of their
proximity functions and we provide several examples. For other examples of
proximity operators with closed-form expression, we refer the
readers to \cite{combettesbook2010} for details.

\begin{exa}\label{example1}
Let $\lambda >0$ and $u\in R^{N}$, then
$$
[prox_{\lambda \| \cdot \|_{1}}(u)]_{i} =
\max(|u_i|-\lambda,0)sign(u_i),
$$
\end{exa}
The proximity operator of $\ell_{1}$-norm $\|\cdot\|_{1}$ is often
referred to as a soft-thresholding operator, and denoted by $Soft(u,\lambda)$,
i.e., $Soft(u,\lambda) = prox_{\lambda \| \cdot \|_{1}}(u)$.

\begin{exa}\label{example2}
Let $\lambda >0$ and $u\in R^{N}$, then
$$
prox_{\lambda \| \cdot \|_{2}}(u) = \max( \|u\|_{2}
-\lambda,0)\frac{u}{\|u\|_{2}}.
$$
\end{exa}

\begin{exa}
Let $u\in \mathds{R}^{2N}$, then the norm $\|u\|_{1,2}$ is defined by
$$
\|u\|_{1,2} = \sum_{i=1}^{N} \sqrt{u_{i}^{2}+u_{N+i}^{2}}.
$$
Let $\lambda > 0$, $x\in  \mathds{R}^{2N}$ and $\|x_{i}\|_{2}=
\sqrt{x_{i}^{2}+x_{N+i}^{2}}$, then $prox_{\lambda
\|\cdot\|_{1,2}}(x)$ can be expressed as
\begin{equation}\label{itv}
\begin{aligned}
\textrm{prox}_{\lambda \|\cdot\|_{1,2}}(x) & =  \Big[
\max \{ \|x_{i}\|_{2} - \lambda , 0   \}\frac{x_{i}}{\|x_{i}\|_{2}} ; \\
\quad &  \max \{ \|x_{i}\|_{2} - \lambda , 0
\}\frac{x_{N+i}}{\|x_{i}\|_{2}}   \Big], \ i = 1, 2, \cdots, N.
\end{aligned}
\end{equation}
\end{exa}

We also prove some proximity functions which will be used in the
following sections.

\begin{lem}\label{lemma1}
For any $u\in R^{N}$ and $b\in R^{N}$, define the function $f(x) =
\|x-b\|_{1}$, then the proximity operator of $prox_{\lambda f}(u)$
is given by
\begin{equation}
prox_{\lambda f}(u) = b + Soft(u-b,\lambda).
\end{equation}
\end{lem}

\begin{proof}
By the definition of the proximity operator, we know that
$$
prox_{\lambda f}(u) = \arg \min_{x} \left\{
\frac{1}{2}\|x-u\|_{2}^{2} + \lambda \|x-b\|_{1} \right\}.
$$
Let $x-b = y$, then the above minimization problem reduces to
\begin{align*}
& \quad \arg \min_{y} \left\{ \frac{1}{2} \| y + b - u \|_{2}^{2} +
\lambda \|y\|_{1} \right\} \\
& = \arg \min_{y} \left\{  \frac{1}{2}\| y - (u-b)  \|_{2}^{2} +
\lambda \|y\|_{1} \right\} \\
& = soft(u-b,\lambda).
\end{align*}
Then, $prox_{\lambda f}(u) = b + soft(u-b,\lambda)$.

\end{proof}

\begin{lem}\label{lemma2}
For any $u\in R^{N}$ and $b\in R^{N}$, define the function $f(x) =
\frac{1}{2} \|x-b\|_{2}^{2}$; then, the proximity operator of
$prox_{\lambda f}(u)$ is given by
\begin{equation}
prox_{\lambda f}(u) = \frac{u+\lambda b}{1+\lambda}.
\end{equation}
\end{lem}

\begin{proof}
By the definition of the proximity operator, we know that
\begin{equation}\label{2normproxi}
prox_{\lambda f}(u) = \arg \min_{x} \left\{
\frac{1}{2}\|x-u\|_{2}^{2} + \frac{1}{2}\lambda \|x-b\|_{2}^{2}
\right\}.
\end{equation}
The first-order optimality condition of (\ref{2normproxi}) reduces
to
$$
0 = (x-u) + \lambda (x-b),
$$
Then $x=\frac{u+\lambda b}{1+\lambda}$. That is,
$$
prox_{\lambda f}(u) = \frac{u+\lambda b}{1+\lambda}.
$$

\end{proof}

Similarly, by Example \ref{example2}, we can deduce the proximity
operator of function $f(x) = \lambda \|x-b\|_{2}$ that
\begin{align}
prox_{\lambda \|\cdot - b\|_{2}}(u) & = \arg\min_{x}\ \left \{
\frac{1}{2}\|x-u\|_{2}^{2} + \lambda \|x-b\|_{2} \right\} \nonumber
\\
& = \max \left \{ \|u-b\|_{2} - \lambda, 0
\right\}\frac{u-b}{\|u-b\|_{2}} + b.
\end{align}

Recall that the Fenchel conjugate of a given function $f$ is defined
as $f^{*}(x) = \sup_{u}\{ <x,u> - f(u)  \}$. The proximity operator
of a function $f$ and its Fenchel conjugate $f^{*}$ are connected by
the celebrated Moreau's identity\cite{Moreau1962}:
\begin{equation}\label{moreau-identity}
x = prox_{\lambda f}(x) + \lambda
prox_{\frac{1}{\lambda}f^{*}}(\frac{x}{\lambda}).
\end{equation}

The well-known Rudi-Osher-Fatemi (ROF)\cite{rof1992} total variation
model is one of the most popular image denoising models.
The ROF model is given by
\begin{equation}
\arg \min_{x} \left \{ \frac{1}{2}\|x-u\|_{2}^{2} + \lambda
\|x\|_{TV} \right\},
\end{equation}
where $u\in R^{d}$ denotes the noisy image and $\|x\|_{TV}$ is the
total variation of $x$. Because total variation regularization can
preserve the edges of images, it has been widely used in the image
restoration and image reconstruction fields. The total variation
norm $\|x\|_{TV}$ can be viewed as the combination of a convex
function with a linear transformation. In fact, let $B$ denote an
$N\times N$ matrix defined by the following:
$$
B := \left(
                                 \begin{array}{cccc}
                                   -1 & 1 &  &  \\
                                    & \ddots & \ddots &  \\
                                    &  & -1 & 1 \\
                                    &  &  & 0 \\
                                 \end{array}
                               \right),
$$
and define matrix $D$ to be $2N^{2}\times N^{2}$, which could be
seen as a finite difference discretization of an image from
historiza and verti,
\begin{equation}
D :=\left(
      \begin{array}{c}
        I \otimes B \\
        B \otimes I \\
      \end{array}
    \right),
\end{equation}
where $I$ is the $N\times N$ identity matrix and the notation
$P\otimes Q$ denotes the Kronecker product of matrices $P$ and $Q$.

Let $x$ be an image in $R^{N^{2}}$. Two definitions of
total variation have appeared in the literature. The first is
referred to as \textit{anisotropic total variation} (ATV) and is
defined by the formula
\begin{equation}\label{atv}
\|x\|_{TV} := \varphi(Dx) =  \|Dx\|_{1},
\end{equation}
where $\varphi(z) :=\|z\|_{1}, z\in R^{2N^{2}}$, whereas the second
definition of total variation is known as \textit{isotropic total
variation } (ITV) and is defined by the equation
\begin{equation}\label{itv-matix}
\|x\|_{TV} = \varphi(Dx)= \|Dx\|_{1,2} = \sum_{i=1}^{n} \left\|
\left(
                                                 \begin{array}{c}
                                                   (Dx)_i \\
                                                   (Dx)_{n+i} \\
                                                 \end{array}
                                               \right)
 \right\|_{2},
\end{equation}
where $\varphi : R^{2N^{2}} \rightarrow R$ as
\begin{equation}
\varphi(z) := \sum_{i=1}^{N^{2}} \left\| \left(
                                      \begin{array}{c}
                                        z_i \\
                                        z_{N^{2}+i} \\
                                      \end{array}
                                    \right)
  \right\|_{2}, z\in R^{2N^{2}}.
\end{equation}

\section{A Primal-dual Proximity Algorithm for Solving
(\ref{problem2})}\label{section3}

In this section, we recall selected primal-dual proximity algorithms for
solving problem (\ref{problem2}). First, the corresponding dual
optimization problem of (\ref{problem2}) is
\begin{equation}\label{problem2-dual}
\max_{y}\ - F^{*}(y) - G^{*}(-K^{*}y).
\end{equation}
Here, $F^{*}$ and $G^{*}$ represent the Fenchel conjugate of $F$ and
$G$, respectively. Combining the primal problem (\ref{problem1}) and
dual problem (\ref{problem2-dual}) leads to the following
saddle-point problem:
\begin{equation}\label{problem2-primal-dual}
\min_{x}\max_{y}\ \langle Kx,y  \rangle + G(x) - F^{*}(y).
\end{equation}
Let problem (\ref{problem2-primal-dual}) have a solution
$(\widehat{x},\widehat{y})$, then it satisfies the following
variational inclusion
\begin{equation}
\begin{aligned}
\left(
  \begin{array}{c}
    0 \\
    0 \\
  \end{array}
\right)\in \left(
             \begin{array}{c}
               K^{*}\widehat{y} + \partial G(\widehat{x}), \\
               - K\widehat{x}+ \partial F^{*}\widehat{y}, \\
             \end{array}
           \right)
\end{aligned}
\end{equation}
where $\partial F^{*}$ and $\partial G$ are the subgradients of the
convex functions $F^{*}$ and $G$.

Chambolle and Pock\cite{chambolleandpock2011} proposed a primal-dual
proximity algorithm for solving (\ref{problem2-primal-dual}). Choosing
$(x^{0},y^{0})\in X\times Y$ and $\overline{x}^{0} = x^{0}$, the
iterative sequences $\{x^{k}\}$ and $\{y^{k}\}$ are given by
\begin{equation}\label{cp-iter}
\left\{
\begin{aligned}
 y^{k+1} & = prox_{\sigma F^{*}}(y^{k}+\sigma K
\overline{x}^{k}), \\
x^{k+1} & = prox_{\tau G}(x^{k} - \tau K^{*}y^{k+1} ), \\
\overline{x}^{k+1} & = x^{k+1} + \theta (x^{k+1}- x^{k}),
\end{aligned}\right.
\end{equation}
where $\sigma, \tau >0$, and $\theta \in [0,1]$. They proved its
convergence with the requirement of $\theta =1$ and $\sigma
\tau < 1/\|K\|^2$ in finite dimensional spaces.

Define $y^{1} = prox_{\sigma F^{*}}(y^{0}+\sigma K x^{0}) $, then
the iterative sequence (\ref{cp-iter}) can be rewritten as
\begin{equation}\label{cp-iter2}
\left\{
\begin{aligned}
x^{k+1} & = prox_{\tau G}(x^{k} - \tau K^{*}y^{k+1} ), \\
y^{k+2} & = prox_{\sigma F^{*}}(y^{k+1} + \sigma K (x^{k+1}  +
\theta (x^{k+1}-x^{k}) )).
\end{aligned}\right.
\end{equation}
Letting $y^{k+1} = y^{k}$, we can simply rewrite the iterative sequence
(\ref{cp-iter2}) as follows
\begin{equation}\label{cp-iter3}
\left\{
\begin{aligned}
x^{k+1} & = prox_{\tau G}(x^{k} - \tau K^{*}y^{k} ), \\
y^{k+1} & = prox_{\sigma F^{*}}(y^{k} + \sigma K (x^{k+1}  + \theta
(x^{k+1}-x^{k}) )).
\end{aligned}\right.
\end{equation}
The only difference between iterative sequences (\ref{cp-iter}) and
(\ref{cp-iter3}) is the initial value of $y^{0}$. Because these
iterative algorithms do not depend on the initial value of $x^{0}$
and $y^{0}$, they are actually equivalent. Therefore, the
details of the primal-dual proximity algorithm introduced by Chambolle
and Pock\cite{chambolleandpock2011} are actually those provided in Algorithm
\ref{primal-dual}.

\begin{algorithm}[H]
 \caption{Primal-dual proximity algorithm for solving (\ref{problem2})}
\begin{algorithmic}\label{primal-dual}
\STATE \textbf{Initialization:} Give $\tau, \sigma >0$, $\theta \in
[0,1]$ and choose $(x^{0},y^{0})\in X\times Y$;
 \STATE For $k=0, 1, 2, \cdots$ do
 \STATE 1.
$x^{k+1}  = prox_{\tau G}(x^{k} - \tau K^{*}y^{k} )$,
 \STATE 2. $y^{k+1}  = prox_{\sigma F^{*}}(y^{k} + \sigma K (x^{k+1}  + \theta
(x^{k+1}-x^{k}) ))$.
 \STATE end
for when some stopping criterion is satisfied
\end{algorithmic}
\end{algorithm}

\begin{thm}(\cite{chambolleandpock2011})\label{theorem1}
Let $\theta =1$ and the parameters $\sigma, \tau$ satisfy $\sigma
\tau \|K\|^{2} < 1$. Then, the sequence $(x^{k},y^{k})$ generated by
Algorithm \ref{primal-dual} converges weakly to an optimal solution
$(x^{*},y^{*})$ of the saddle-point problem
(\ref{problem2-primal-dual}).
\end{thm}

\begin{rmk}
Condat\cite{condat2013} proved that the  condition $\sigma \tau
\|K\|^{2} < 1$ in Theorem could be relaxed to $\sigma \tau \|K\|^{2}
\leq 1$ in finite dimensional spaces.

\end{rmk}

The convergence of Algorithm \ref{primal-dual} relies on the
operator norm $\|K\|$, which is not easy to estimate. Pock and
Chambolle\cite{pock1} attempted to address this shortcoming by
proposing a precondition technique for Algorithm \ref{primal-dual}
where the step sizes $\tau$ and $\sigma$ are replaced by two
symmetric and positive definite matrices, respectively. They also
suggested a practical approach for obtaining the matrices
$\mathrm{T}$ and $\mathrm{\Sigma}$, thereby satisfying the
convergence requirement of Theorem \ref{theorem2}.

\begin{algorithm}[H]
 \caption{Preconditioned primal-dual proximity algorithm for solving (\ref{problem2})}
\begin{algorithmic}\label{preconditioned-primal-dual}
\STATE \textbf{Initialization:} Choose symmetric and positive
definite matrices $\mathrm{T}$ and $\mathrm{\Sigma}$, $\theta \in
[0,1]$, $(x^{0}, y^{0})\in X\times Y$.
 \STATE For $k=0, 1, 2, \cdots$ do
 \STATE 1.
$x^{k+1}  = prox_{\mathrm{T} G}(x^{k} - \mathrm{T} K^{*}y^{k} )$,
 \STATE 2. $y^{k+1}  = prox_{\Sigma F^{*}}(y^{k} + \Sigma K (x^{k+1}  + \theta
(x^{k+1}-x^{k}) ))$.
 \STATE end
for when some stopping criterion is satisfied
\end{algorithmic}
\end{algorithm}

\begin{thm}(\cite{pock1})\label{theorem2}
Let $\theta =1$ and let $\mathrm{T}, \mathrm{\Sigma}$ be symmetric
and positive definite matrices such that
$\|\mathrm{\Sigma}^{\frac{1}{2}} K \mathrm{T}^{\frac{1}{2}}\|<1$.
Then, the sequence $(x^{k},y^{k})$ generated by Algorithm
\ref{preconditioned-primal-dual} converges weakly to an optimal
solution $(x^{*},y^{*})$ of the saddle-point problem
(\ref{problem2-primal-dual}).
\end{thm}

As mentioned in \cite{pock1}, the matrices $\mathrm{\Sigma}$ and
$\mathrm{T}$ could be any symmetric and positive matrices. However, it
is a prior requirement of Algorithm
\ref{preconditioned-primal-dual} that the proximity operators are
simple. Thus, they proposed to choose $\mathrm{\Sigma}$ and
$\mathrm{T}$ with some diagonal matrices which satisfy all these
requirements and guarantee the convergence of the algorithm.

\begin{lem}\label{lemema-cp}(\cite{pock1})
Let $\mathrm{T} = diag(\tau)$, where $\tau = (\tau_1, \tau_2,
\cdots, \tau_n)$ and $\mathrm{\Sigma}=diag(\sigma)$, where $\sigma =
(\sigma_1, \cdots, \sigma_m)$. In particular,
$$
\tau_j = \frac{1}{\sum_{i=1}^{m}|K_{i,j}|^{2-\alpha}}, \quad
\sigma_i = \frac{1}{\sum_{j=1}^{n}|K_{i,j}|^{\alpha}},
$$
then for any $\alpha \in [0,2]$,
$$
\| \mathrm{\Sigma}^{\frac{1}{2}}  K  \mathrm{T}^{\frac{1}{2}} \|^2 =
\sup_{x\in X, x\neq 0} \frac{\| \mathrm{T}^{\frac{1}{2}} K
\mathrm{\Sigma}^{\frac{1}{2}}x \|^2  }{\|x\|^2}\leq 1.
$$
\end{lem}

In the next section, we shall see how to judiciously use the
primal-dual proximity algorithms, including Algorithm
\ref{primal-dual} and Algorithm \ref{preconditioned-primal-dual}, to
derive a variety of flexible convex optimization algorithms for the
proposed problem (\ref{problem1}).

\section{A Splitting Primal-dual Proximity Algorithm for Solving
(\ref{problem1})}\label{section4}

In comparison with the well-known forward-backward splitting
algorithm and the alternating direction method of multipliers for
solving problem (\ref{problem2}), the forward-backward splitting
algorithm needs one of the functions $F$ or $G$ to satisfy the
differential and requires a Lipschitz continuous gradient, and the
alternating direction method of multipliers always involves a system
of linear equations as its subproblem. In contrast, every subproblem of the
primal-dual proximity algorithm is easy to solve and does not require any inner
iteration numbers. This motivated us to extend the primal-dual
proximity algorithm to solve the general optimization problem
(\ref{problem1})

First, we present the main iterative algorithm to solve problem
(\ref{problem1}) and prove its convergence as follows.

\begin{algorithm}[H]
 \caption{A splitting primal-dual proximity algorithm for solving (\ref{problem1})}
\begin{algorithmic}\label{splitting-primal-dual}
\STATE \textbf{Initialization:} Give $\tau, \sigma >0$ such that
$\tau \sigma < 1/\|\sum_{i=1}^{l}K_{i}^{*}K_i\|$, choose
$(x^{0},y_{1}^{0}, y_{2}^{0}, \cdots, y_{l}^{0})\in X\times
Y_{1}\times Y_{2}\times \cdots \times Y_{l}$;
 \STATE For $k=0, 1, 2, \cdots$ do
 \STATE 1.
$x^{k+1} = prox_{\tau G}(x^{k} - \tau
\sum_{i=1}^{l}K_{i}^{*}y_{i}^{k} ) $,
 \STATE 2. $y_{i}^{k+1} = prox_{\sigma F_{i}^{*}}(y_{i}^{k}+\sigma K_{i}
(2x^{k+1}-x^{k}))$, for $i=1, 2, \cdots, l$.
 \STATE end
for when some stopping criterion is satisfied
\end{algorithmic}
\end{algorithm}

The dual problem of (\ref{problem1}) is
\begin{equation}
\max_{y_{1}, \cdots, y_l}\ - G^{*}\left( -\sum_{i=1}^{l}K_{i}^{*}y_i
\right) - \sum_{i=1}^{l}F_{i}^{*}(y_i),
\end{equation}
and the saddle-point problem is
\begin{equation}\label{problem1-primal-dual}
\min_{x}\max_{y_1, \cdots, y_l}\ \sum_{i=1}^{l}\langle K_i x, y_i
\rangle + G(x) - \sum_{i=1}^{l}F_{i}^{*}(y_i).
\end{equation}

\begin{thm}
Let $\sigma>0$ and $\tau>0$ be the parameters of Algorithm
\ref{splitting-primal-dual}, then the iterative sequence $(x^{k},
y_{1}^{k}, \cdots, y_{l}^{k})$ converges weakly to an optimal
solution $(x^{*}, y_{1}^{*}, \cdots, y_{l}^{*})$ of the saddle-point
problem (\ref{problem1-primal-dual}).
\end{thm}

\begin{proof}
First, we convert the optimization problem (\ref{problem1}) into the
form of problem (\ref{problem2}) by using  a product spaces
technique. For this purpose, we introduce the notation $\textbf{y}
:= (y_1, \cdots, y_l)$ for an element of the Hilbert space $
\textit{Y} := \textit{Y}_{1}\times \cdots \textit{Y}_l$, equipped
with the inner product $\langle \textbf{y}, \textbf{z} \rangle =
\sum_{i=1}^{l}\langle y_i, z_i \rangle$. For any $\textbf{y}\in
\textit{Y}$, we define the function $F\in \Gamma_{0}(Y)$ by
$\widetilde{\textbf{F}}(\textbf{y}) = \sum_{i=1}^{l}F_{i}(y_i)$ and
the linear function $\widetilde{\textbf{K}}: X\rightarrow
\textit{Y}$ by $\widetilde{\textbf{K}}x := (K_1 x, \cdots, K_{l}x)$,
i.e.,
$$
\widetilde{\textbf{K}} = \left(
                  \begin{array}{c}
                    K_1 \\
                    K_2 \\
                    \vdots \\
                    K_l \\
                  \end{array}
                \right).
$$
Then, we know that $(\widetilde{\textbf{F}} \circ
\widetilde{\textbf{K}})(x) = F_1 (K_1 x) + F_2 (K_2 x) + \cdots +
F_l (K_1 x)$. Therefore, the optimization problem (\ref{problem1})
can be reformulated as the following
$$
\min_{x}\ (\widetilde{\textbf{F}} \circ \widetilde{\textbf{K}})(x) +
G(x),
$$
which is the exact optimization problem (\ref{problem2}). Taking
$\theta =1$ in Algorithm \ref{primal-dual}, we obtain the iterative
sequence for solving (\ref{problem1}).
\begin{equation}\label{compact-primal-dual}
\left \{
\begin{aligned}
x^{k+1} & = prox_{\tau G}(x^{k}-\tau \widetilde{\textbf{K}}^{*}\textbf{y}^{k}), \\
\textbf{y}^{k+1} & = prox_{\sigma
\widetilde{\textbf{F}}^{*}}(\textbf{y}^{k}+ \sigma
\widetilde{\textbf{K}} (2x^{k+1}-x^{k})).
\end{aligned}\right.
\end{equation}
By Theorem \ref{theorem1}, we can conclude that the iterative
sequence $(x^{k}, y_{1}^{k}, \cdots, y_{l}^{k})$ converges weakly to
an optimal solution $(x^{*}, y_{1}^{*}, \cdots, y_{l}^{*})$ of the
saddle-point problem (\ref{problem1-primal-dual}). Further, as the
function $\widetilde{\textbf{F}}$ is separable with variables, the
Fenchel conjugate of $\widetilde{\textbf{F}}^{*}(\textbf{u}) =
F_{1}^{*}(u_1) + F_{2}^{*}(u_2) + \cdots + F_{l}^{*}(u_l)$, for
$\textbf{u} := (u_1, u_2, \cdots, u_l)\in \textit{Y}$. Then, the
proximity operator $prox_{\sigma \widetilde{\textbf{F}}^{*}}$ can be
calculated independently, i.e., $prox_{\sigma
\widetilde{\textbf{F}}^{*}}(\textbf{u}) = (prox_{\sigma
F_{1}^{*}}(u_1), \cdots, prox_{\sigma F_{l}^{*}}(u_l))$. Therefore,
we can split the iterative sequence (\ref{compact-primal-dual}) and
obtain the corresponding  Algorithm \ref{splitting-primal-dual} as
stated before.

\end{proof}

\begin{rmk}
Based on the results of Condat, the parameters can be relaxed to $\tau
\sigma \leq 1/ \|\sum_{i=1}^{l}K_{i}^{T}K_{i}\|$ in a finite
dimensional Hilbert space.
\end{rmk}

\begin{algorithm}[H]
 \caption{A preconditioned splitting primal-dual proximity algorithm for solving (\ref{problem1})}
\begin{algorithmic}\label{preconditioned-splitting-primal-dual}
\STATE \textbf{Initialization:} Choose symmetric and positive
definite matrices $\mathrm{T}$ and $\mathrm{\Sigma_{i}}$, for $i=1,
2, \cdots, l$,
 $(x^{0},y_{1}^{0}, y_{2}^{0},
\cdots, y_{l}^{0})\in X\times Y_{1}\times Y_{2}\times \cdots \times
Y_{l}$;
 \STATE For $k=0, 1, 2, \cdots$ do
 \STATE 1.
$x^{k+1} = prox_{\mathrm{T} G}(x^{k} - \mathrm{T}
\sum_{i=1}^{l}K_{i}^{*}y_{i}^{k} ) $,
 \STATE 2. $y_{i}^{k+1} = prox_{\mathrm{\Sigma_{i}} F_{i}^{*}}(y_{i}^{k}+\mathrm{\Sigma_{i}}K_{i}
(2x^{k+1}-x^{k}))$, for $i=1, 2, \cdots, l$.
 \STATE end
for when some stopping criterion is satisfied
\end{algorithmic}
\end{algorithm}

Based on Lemma \ref{lemema-cp}, we are able to suggest a practical way to
choose the matrices $\mathrm{T}$ and
$(\mathrm{\Sigma_{k}})_{k=1}^{l}$, respectively.
\begin{lem}\label{ourlemms}
Let $\mathrm{T} = diag(\tau)$, where $\tau = (\tau_1, \tau_2,
\cdots, \tau_n)$ and $\mathrm{\Sigma_k}=diag(\sigma^{k})$, where
$\sigma^{k} = (\sigma^{k}_{1}, \cdots, \sigma^{k}_{m_k})$, for $k=1,
2, \cdots, l$. In particular,
$$
\tau_j =
\frac{1}{\sum_{k=1}^{l}\sum_{i=1}^{m}|K_{k}(i,j)|^{2-\alpha}}, \quad
\sigma^{k}_{i} = \frac{1}{\sum_{j=1}^{n}|K_{k}(i,j)|^{\alpha}},
$$
then for any $\alpha \in [0,2]$,
$$
\| \mathrm{T}^{\frac{1}{2}} \widetilde{K}
\widetilde{\mathrm{\Sigma}}^{\frac{1}{2}}  \|^2 = \sup_{x\in X,
x\neq 0} \frac{\| \mathrm{T}^{\frac{1}{2}} \widetilde{K}
\widetilde{\mathrm{\Sigma}}^{\frac{1}{2}}x \|^2  }{\|x\|^2}\leq 1,
$$
where $\widetilde{K} = (K_1; K_2; \cdots ; K_l)$ and
$\widetilde{\mathrm{\Sigma}} = (\mathrm{\Sigma}_{1};
\mathrm{\Sigma}_{2}; \cdots; \mathrm{\Sigma}_{l})$.
\end{lem}

\begin{rmk}
The advantages of our approach are the following:

 (i)\
There are limited assumptions for the functions
$\{F_i\}_{i=1}^{l}$ and $G$;

(ii)\ There is no inner iteration involved in the main process;

(iii)\ The iterative parameters are easy to select.
\end{rmk}

\section{Applications}\label{section5}

In this section, we consider solving the following constrained
composite optimization problem,

\begin{equation}\label{constrained-l2-l1-tv}
\begin{aligned}
& \min_{x}\ \frac{1}{2}w_1 \|Ax-b\|_{2}^{2} + w_2 \|Ax-b\|_1 +
\lambda \|x\|_{TV}, \\
& s.t.\ x\in C,
\end{aligned}
\end{equation}
where $x\in R^{n}$, $A\in R^{m\times n}$, $b\in R^{m}$, $C$ is a
closed convex set, $w_1, w_2\in [0,1]$ satisfying $w_1 + w_2 =1$,
$\lambda$ is the regularization parameter, and $\|x\|_{TV}$ denotes
the total variation (TV) norm.

It is easy to see that problem (\ref{constrained-l2-l1-tv}) includes
the well-known $L2+TV$ (\ref{L2-TV}) and $L1+TV$ (\ref{L1+TV})
problem as its special case. If $w_2 = 0$, then it reduces to the
constrained $L2+TV$ problem, and if $w_1 =0$, then it reduces to the
constrained $L1+TV$ problem, respectively.

In the following, we show that the optimization problem
(\ref{constrained-l2-l1-tv}) is a special case of problem
(\ref{problem1}). The flexibility of problem (\ref{problem1}) lies
in the ease with which constraints can be incorporated into this problem.
It is observed from the definition of the total variation semi-norm (\ref{atv})
and (\ref{itv-matix}) that $\|x\|_{TV} = (\varphi \circ D)(x)$, with
$\varphi$ a convex lower semicontinuous function and $D$ a real
matrix. Then, the optimization problem (\ref{constrained-l2-l1-tv})
can be reformulated as follows.
\begin{equation}\label{constrained-l2-l1-tv-equal}
\min_{x}\ \frac{1}{2}w_1 \|Ax-b\|_{2}^{2} + w_2 \|Ax-b\|_1 + \lambda
\varphi (Dx) + \iota_{C}(x),
\end{equation}
where $\iota_{C}$ is the indicator function of the closed convex set
$C$.

To match the formulation (\ref{problem1}) with the problem at hand
(\ref{constrained-l2-l1-tv-equal}), we follow two approaches to obtain
its solution.

\noindent\textbf{Method I.} Let $G(x) = 0$, $F_{1}(v) =
\frac{1}{2}w_1\|v-b\|_{2}^{2}$, $K_1 = A$, $F_{2}(v)= w_2
\|v-b\|_{1}$, $K_2 = A$, $F_3(v)=\lambda \varphi(v)$, $K_3 = D$,
$F_4(v)=\iota_{C}(v)$, and $K_4 = I$. Then, we can apply Algorithm
\ref{splitting-primal-dual} to solve the problem
(\ref{constrained-l2-l1-tv-equal}). The detailed structure of the algorithm
is presented as follows.

\begin{algorithm}[H]
 \caption{A first class of splitting primal-dual proximity algorithm for solving problem (\ref{constrained-l2-l1-tv})}
\begin{algorithmic}\label{first-splitting}
\STATE \textbf{Initialization:} Give $\tau, \sigma >0$ such that
$\tau \sigma \leq 1/\| 2A^{T}A + D^{T}D + I \|$, choose
$(x^{0},y_{1}^{0}, y_{2}^{0}, y_{3}^{0}, y_{4}^{0})\in X\times
Y_{1}\times Y_{2}\times Y_{3} \times Y_{4}$;
 \STATE For $k=0, 1, 2, \cdots$ do
 \STATE 1.
$x^{k+1} = x^{k} - \tau
(A^{T}y_{1}^{k}+A^{T}y_{2}^{k}+D^{T}y_{3}^{k}+y_{4}^{k})$,
 \STATE 2. $y_{1}^{k+1} = prox_{\sigma F_{1}^{*}}(y_{1}^{k}+\sigma A
(2x^{k+1}-x^{k}))$,
 \STATE 3. $y_{2}^{k+1} = prox_{\sigma F_{2}^{*}}(y_{2}^{k}+\sigma A
(2x^{k+1}-x^{k}))$,
 \STATE 4. $y_{3}^{k+1} = prox_{\sigma F_{3}^{*}}(y_{3}^{k}+\sigma D
(2x^{k+1}-x^{k}))$,
 \STATE 5. $y_{4}^{k+1} = prox_{\sigma F_{4}^{*}}(y_{4}^{k}+\sigma
(2x^{k+1}-x^{k}))$.
 \STATE end
for when some stopping criterion is satisfied
\end{algorithmic}
\end{algorithm}

In the following, we explain that every subproblem of Algorithm
\ref{first-splitting} can be calculated explicitly. In fact, the
proximal operator of $F^{*}$ is determined via one of the functions $F$
obtained by using Moreau's identity (\ref{moreau-identity}).

First, according to Moreau's identity (\ref{moreau-identity})
and Lemma \ref{lemma2}, we have
\begin{align}
y_{1}^{k+1} & = prox_{\sigma F_{1}^{*}}(\sigma
(\frac{1}{\sigma}y_{1}^{k} + A(2x^{k+1}-x^{k}))) \nonumber \\
& = \sigma (I -
prox_{\frac{1}{\sigma}F_{1}})(\frac{1}{\sigma}y_{1}^{k}+
A(2x^{k+1}-x^{k})) \nonumber \\
& = (y_{1}^{k}+\sigma A(2x^{k+1}-x^{k})) -
\frac{\sigma}{w_1+\sigma}(y_{1}^{k}+\sigma A(2x^{k+1}-x^{k}) + w_1
b) \nonumber \\
& = \frac{w_1}{w_1 + \sigma} ( y_{1}^{k} + \sigma A (2x^{k+1}-x^{k})
- \sigma b ).
\end{align}

Second, by Lemma \ref{lemma1}, we can obtain the proximity operator
of function $\sigma F_{2}^{*}$. That is
\begin{align}
y_{2}^{k+1} & = prox_{\sigma F_{2}^{*}}(\sigma (\frac{1}{\sigma}
y_{2}^{k} +
A(2x^{k+1}-x^{k}))) \nonumber \\
& = \sigma (I - prox_{\frac{1}{\sigma}F_{2}})(\frac{1}{\sigma}
y_{2}^{k} + A(2x^{k+1}-x^{k})) \nonumber \\
& = (y_{2}^{k}+\sigma A(2x^{k+1}-x^{k})) - \sigma (b +
Soft(\frac{1}{\sigma} y_{2}^{k} + A(2x^{k+1}-x^{k}) - b,
\frac{w_2}{\sigma})).
\end{align}

Third, by taking into account the definition of the TV norm, the function $F_3(v)$ is
equal to $\|v\|_{1}$ or $\|v\|_{1,2}$, respectively. Then, the
proximity of $\sigma F_{3}^{*}$ can also be calculated by
\begin{align}
y_{3}^{k+1} & = prox_{\sigma F_{3}^{*}}(\sigma
(\frac{1}{\sigma}y_{3}^{k} + D(2x^{k+1}-x^{k}) )) \nonumber \\
& = \sigma (I -
prox_{\frac{1}{\sigma}F_{3}})(\frac{1}{\sigma}y_{3}^{k} +
D(2x^{k+1}-x^{k})).
\end{align}
Then, for the anisotropic TV (ATV), we have
\begin{equation}
y_{3}^{k+1} = (y_{3}^{k} + \sigma D(2x^{k+1}-x^{k}) ) - \sigma
Soft(\frac{1}{\sigma}y_{3}^{k} + D(2x^{k+1}-x^{k}),
\frac{\lambda}{\sigma}).
\end{equation}
and for the isotropic TV (ITV), we also have a closed-form solution
due to (\ref{itv}).

Fourth, because the proximity of indicator function $\iota_{C}$ is
equal to the projection operator onto the set $C$, we obtain
\begin{align}
y_{4}^{k+1} & = prox_{\sigma F_{4}^{*}}(\sigma
(\frac{1}{\sigma}y_{4}^{k} + (2x^{k+1}-x^{k}) )) \nonumber \\
& = \sigma (I -
prox_{\frac{1}{\sigma}F_{4}})(\frac{1}{\sigma}y_{4}^{k} +
(2x^{k+1}-x^{k}))\nonumber \\
& = (y_{4}^{k} + \sigma(2x^{k+1}-x^{k})) - \sigma
P_{C}(\frac{1}{\sigma}y_{4}^{k} + (2x^{k+1}-x^{k})).
\end{align}

Therefore, the original problem (\ref{constrained-l2-l1-tv}) is
decomposed into an iterative sequence consisting of subproblems which
are much easier to solve, each one with a closed-form solution.

Next, we follow another approach to solve problem
(\ref{constrained-l2-l1-tv-equal}).

\noindent\textbf{Method II.} Let $G(x) = \iota_{C}(x)$, $F_{1}(v) =
\frac{1}{2}w_1\|v-b\|_{2}^{2}$, $K_1 = A$, $F_{2}(v)= w_2
\|v-b\|_{1}$, $K_2 = A$, $F_3(v)=\varphi(v)$, and $K_3 = D$. Then,
we can apply Algorithm \ref{splitting-primal-dual} to solve the
problem (\ref{constrained-l2-l1-tv-equal}),

\begin{algorithm}[H]
 \caption{A second class of splitting primal-dual proximity algorithm for solving problem (\ref{constrained-l2-l1-tv})}
\begin{algorithmic}\label{second-splitting}
\STATE \textbf{Initialization:} Give $\tau, \sigma >0$ such that
$\tau \sigma \leq \| 2A^{T}A + D^{T}D \|$, choose $(x^{0},y_{1}^{0},
y_{2}^{0}, y_{3}^{0}, y_{4}^{0})\in X\times Y_{1}\times Y_{2}\times
Y_{3} \times Y_{4}$;
 \STATE For $k=0, 1, 2, \cdots$ do
 \STATE 1.
$x^{k+1} = P_{C}(x^{k} - \tau
(A^{T}y_{1}^{k}+A^{T}y_{2}^{k}+D^{T}y_{3}^{k})$,
 \STATE 2. $y_{1}^{k+1} = prox_{\sigma F_{1}^{*}}(y_{1}^{k}+\sigma A
(2x^{k+1}-x^{k}))$,
 \STATE 3. $y_{2}^{k+1} = prox_{\sigma F_{2}^{*}}(y_{2}^{k}+\sigma A
(2x^{k+1}-x^{k}))$,
 \STATE 4. $y_{3}^{k+1} = prox_{\sigma F_{3}^{*}}(y_{3}^{k}+\sigma D
(2x^{k+1}-x^{k}))$.
 \STATE end
for when some stopping criterion is satisfied
\end{algorithmic}
\end{algorithm}

\begin{rmk}
(1)\ The difference between Algorithm \ref{first-splitting} and
Algorithm \ref{second-splitting} is that they treat the constraint
$C$ differently. In Algorithm \ref{first-splitting}, the indicator
function is set as the combination of a convex function with an identity
matrix, whereas in Algorithm \ref{second-splitting}, the indicator
function is defined as the function $G(x)$ in problem
(\ref{problem1}).

(2)\ Algorithm \ref{first-splitting} and Algorithm
\ref{second-splitting} use a fixed step size, which depends on the
estimation of some matrix norm. This norm is its largest singular
value, which can be computed via the power method in practice.
\end{rmk}

Based on the preconditioned splitting primal-dual proximity
algorithm (Algorithm \ref{preconditioned-splitting-primal-dual}), we
obtain the corresponding preconditioned Algorithm
\ref{first-splitting} and Algorithm \ref{second-splitting},
respectively.

\begin{algorithm}[H]
 \caption{A first class of preconditioned splitting primal-dual proximity algorithm for solving problem (\ref{constrained-l2-l1-tv})}
\begin{algorithmic}\label{first-pre-splitting}
\STATE \textbf{Initialization:} Follow the Lemma to define the
matrices $\mathrm{T}$ and $\mathrm{\Sigma_{i}}$, $i=1 ,2 , 3, 4$;
Choose $(x^{0},y_{1}^{0}, y_{2}^{0}, y_{3}^{0}, y_{4}^{0})\in
X\times Y_{1}\times Y_{2}\times Y_{3} \times Y_{4}$;
 \STATE For $k=0, 1, 2, \cdots$ do
 \STATE 1.
$x^{k+1} = x^{k} - \mathrm{T}
(A^{T}y_{1}^{k}+A^{T}y_{2}^{k}+D^{T}y_{3}^{k}+y_{4}^{k})$,
 \STATE 2. $y_{1}^{k+1} = prox_{\mathrm{\Sigma_{1}}
 F_{1}^{*}}(y_{1}^{k}+\mathrm{\Sigma_{1}} A
(2x^{k+1}-x^{k}))$,
 \STATE 3. $y_{2}^{k+1} = prox_{\mathrm{\Sigma_{2}}
 F_{2}^{*}}(y_{2}^{k}+\mathrm{\Sigma_{2}} A
(2x^{k+1}-x^{k}))$,
 \STATE 4. $y_{3}^{k+1} = prox_{\mathrm{\Sigma_{3}}
 F_{3}^{*}}(y_{3}^{k}+\mathrm{\Sigma_{3}} D
(2x^{k+1}-x^{k}))$,
 \STATE 5. $y_{4}^{k+1} = prox_{\mathrm{\Sigma_{4}}
 F_{4}^{*}}(y_{4}^{k}+\mathrm{\Sigma_{4}}
(2x^{k+1}-x^{k}))$.
 \STATE end
for when some stopping criterion is satisfied
\end{algorithmic}
\end{algorithm}

Similarly, we can provide preconditioned Algorithm
\ref{second-splitting} as follows.

\begin{algorithm}[H]
 \caption{A second class of preconditioned splitting primal-dual proximity algorithm for solving problem (\ref{constrained-l2-l1-tv})}
\begin{algorithmic}\label{second-pre-splitting}
\STATE \textbf{Initialization:} Follow the Lemma to define the
matrices $\mathrm{T}$ and $\mathrm{\Sigma_{i}}$, $i=1 ,2 , 3$;
Choose $(x^{0},y_{1}^{0}, y_{2}^{0}, y_{3}^{0})\in X\times
Y_{1}\times Y_{2}\times Y_{3}$;
 \STATE For $k=0, 1, 2, \cdots$ do
 \STATE 1.
$x^{k+1} = P_{C}(x^{k} - \mathrm{T}
(A^{T}y_{1}^{k}+A^{T}y_{2}^{k}+D^{T}y_{3}^{k})$,
 \STATE 2. $y_{1}^{k+1} = prox_{\mathrm{\Sigma_{1}}
 F_{1}^{*}}(y_{1}^{k}+\mathrm{\Sigma_{2}} A
(2x^{k+1}-x^{k}))$,
 \STATE 3. $y_{2}^{k+1} = prox_{\mathrm{\Sigma_{2}}
 F_{2}^{*}}(y_{2}^{k}+\mathrm{\Sigma_{2}} A
(2x^{k+1}-x^{k}))$,
 \STATE 4. $y_{3}^{k+1} = prox_{\mathrm{\Sigma_{3}}
 F_{3}^{*}}(y_{3}^{k}+\mathrm{\Sigma_{3}} D
(2x^{k+1}-x^{k}))$.
 \STATE end
for when some stopping criterion is satisfied
\end{algorithmic}
\end{algorithm}

\begin{rmk}

(1)\ For the unconstrained optimization problem
(\ref{constrained-l2-l1-tv}), i.e., $C :=R^{n}$, Algorithm
\ref{first-splitting} and Algorithm \ref{second-splitting} are
equivalent, as are Algorithm \ref{first-pre-splitting} and
Algorithm \ref{second-pre-splitting}.

(2)\ In comparison with Algorithm \ref{first-splitting} and
Algorithm \ref{second-splitting}, Algorithm
\ref{first-pre-splitting} and Algorithm \ref{second-pre-splitting}
can be used to obtain the iterative parameters self-adaptively without the need
to know the respective matrix norm.

\end{rmk}

\section{Numerical experiments}\label{section6}

In Section \ref{section5}, we derived an instance of the proposed
splitting primal-dual proximity algorithms.
 To demonstrate the performance of these proposed algorithms, we apply
 them to the test problems described in Section \ref{section5}.
 All experiments were performed using MATLAB on a
 Lenovo Thinkstation running Windows 7 with an Intel Core 2 CPU and 4 GB of RAM.

Two-dimensional tomography test problems were created by using
AIRTools\cite{hansen1}, which is a MATLAB software package for tomographic
reconstruction that was developed by Prof. Perchristian Hansen and
his collaborators. The package includes two core functions "fanbeamtomo" and
"paralleltomo", which were used to generate the simulation data. For
example, the function "paralleltomo" creates a 2D tomography test
problem using parallel beams.
\begin{equation}\label{paralletomo}
[A,b,x] = paralleltomo(N,theta,p),
\end{equation}
where the input variables are as follows: $N$ is a scalar denoting
the number of discretization intervals in each dimension such that
the domain consists of $N^{2}$ cells, $theta$ is a vector containing
the angles in degrees (default: $theta = 0:1:179$), and $p$ is the number
of parallel rays for each angle (default: $p=round(\sqrt{2}N)$). The output
variables are the following: $A$ is a coefficient matrix with $N^{2}$ columns and
$length(theta)*p$ rows, $b$ is a vector containing the projection
data, and $x$ is a vector containing the exact solution with elements
between $0$ and $1$. We refer the reader to the AIRTools manual for
further details. The test image is the standard benchmark
Shepp-Logan phantom (see Figure \ref{fig1}.) with size $256\times
256$ and pixels are assigned values varying from $0$ to $1$.

\begin{figure} %[htbp]
\centering \setlength{\floatsep}{0pt}
\setlength{\abovecaptionskip}{-20pt}
\scalebox{0.4}{\includegraphics[width=1\textwidth]{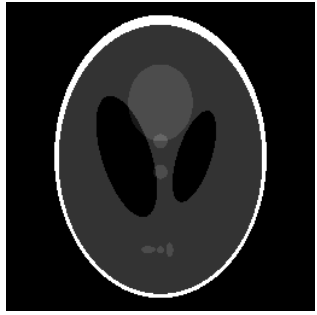}}
\caption[]{Original Shepp-Logan phantom}\label{fig1}
\end{figure}

% We add both Gaussian noise and impulsive noise to the observed data
% vector $b$.

We measured the quality of recovered images by using the criterion
signal-to-noise ratio (SNR),
$$
SNR = log_{10}\left (  \frac{\|x_{true}\|_{2}^{2}}{ \|x_{true} -
x_{rec}\|_{2}^2 } \right),
$$
where $x_{true}$ is the original image, $x_{rec}$ denotes the
reconstructed image obtained by using the iterative algorithms. The
iterative process is stopped when the relative error
$$
\frac{\|x^{k+1}-x^{k}\|_{2}}{\|x^{k}\|_{2}} \leq \epsilon,
$$ where $\epsilon$ is a
given small real number.

We compare the performance of Algorithm \ref{first-splitting} and
Algorithm \ref{second-splitting}, as well as that of Algorithm
\ref{first-pre-splitting} and Algorithm \ref{second-pre-splitting}.
The anisotropic TV (ATV) and isotropic TV (ITV) perform similarly;
therefore, we use the (ATV) regularization term in the
following test. The initial values of the variables are set to zero in
all iterative algorithms. In Algorithm \ref{first-splitting} and
Algorithm \ref{second-splitting}, the induced norm of the operator
$\| 2A^{T}A + D^{T}D + I \|$ and $\| 2A^{T}A + D^{T}D \|$ are
estimated using the standard power iteration algorithm. For the
preconditioned algorithm, the parameter $\alpha$ was set to one in
Lemma \ref{ourlemms}.

The projection angles in (\ref{paralletomo}) are set as $theta =
0:10:179$. A total of 18 angles were used in the simulation
test. The $p$ value is set as default. Then, the system matrix $A$ is
$6516\times 65536$, which is an under-determined matrix. Both
Gaussian and impulsive noise are added to the projection data vector
$b$. The performance of Algorithm \ref{first-splitting}, Algorithm
\ref{second-splitting}, Algorithm \ref{first-pre-splitting},
  and Algorithm \ref{second-pre-splitting} are listed in Table \ref{first-nonnegative}
and Table \ref{first-bounded}, respectively. The '-' entries
indicate that the algorithm failed to reduce the error below the given
tolerance $\epsilon$ within a maximum number of $40000$ iterations.

% The process is listed below:
% \begin{align}
% \textit{Gaussian noise}:\ e & = randn(size(b)), \nonumber \\
% e & = e/norm(e), bn = b + 0.01*norm(b)*e, \nonumber \\
% \textit{Impulsive noise}:\ noise & = rand(length(b),1), \nonumber \\
% ImpulseNoiseProbability & = 1/75, \nonumber \\
% noise & = 20 *(noise < ImpulseNoiseProbability) - 10, \nonumber \\
% bn & = max(bn,noise).
% \end{align}

\begin{table}[htbp]
\footnotesize
  \centering
  \caption{Comparison of the performance of Algorithms \ref{first-splitting}, \ref{second-splitting}, \ref{first-pre-splitting},
  and \ref{second-pre-splitting} in terms of SNR and iteration numbers with non-negativity constraints, i.e., $C= \{ x | x\geq
  0\}$}.
    \begin{tabular}{cccccc}
   \hline
\multirow{2}[1]{*}{Regularization} & \multirow{2}[1]{*}{Methods} &    $\epsilon = 10^{-3}$    &  $\epsilon = 10^{-4}$      &   $\epsilon = 10^{-5}$      &  $\epsilon = 10^{-6}$    \\
     Parameter      & & $SNR (dB)/k$  &  $SNR(dB)/k$ & $SNR(dB)/k$ & $SNR(dB)/k$  \\
    \hline
    \multirow{4}[1]{*}{$\lambda=0.6$} & Algorithm \ref{first-splitting} & $18.66/3253$ & $24.31/15804$ & $25.25/36598$ & $25.34/-$   \\
  & Algorithm \ref{first-pre-splitting} & $24.40/430$ & $25.75/1236$ & $26.17/5852$ & $26.18/21265$   \\
  & Algorithm \ref{second-splitting} & $19.08/1882$ & $25.65/14659$ & $26.11/30009$ & $26.15/-$  \\
    & Algorithm \ref{second-pre-splitting} & $24.52/378$ & $25.86/1154$ & $26.17/4634$ & $26.18/16433$   \\
    \hline
\multirow{4}[1]{*}{$\lambda=0.8$} & Algorithm \ref{first-splitting} & $18.73/3399$ & $25.49/17359$ & $26.73/38831$ & $26.78/-$   \\
  & Algorithm \ref{first-pre-splitting} & $25.33/465$ & $27.11/1346$ & $27.67/5805$ & $27.73/22503$   \\
  & Algorithm \ref{second-splitting} & $19.01/2158$ & $26.84/17613$ & $27.53/36010$ & $27.55/-$  \\
    & Algorithm \ref{second-pre-splitting} & $25.43/404$ & $27.27/1286$ & $27.67/4559$ & $27.74/16967$   \\
    \hline
\multirow{4}[1]{*}{$\lambda=1.2$} & Algorithm \ref{first-splitting} & $18.67/3754$ & $27.41/19697$ & $28.87/-$ & $28.87/-$   \\
  & Algorithm \ref{first-pre-splitting} & $26.57/495$ & $29.10/1473$ & $29.86/4462$ & $30.09/21411$   \\
  & Algorithm \ref{second-splitting} & $18.31/2593$ & $28.29/20614$ & $29.47/-$ & $29.47/-$  \\
    & Algorithm \ref{second-pre-splitting} & $26.65/445$ & $29.13/1412$ & $29.87/3930$ & $30.11/19898$   \\
    \hline
\multirow{4}[1]{*}{$\lambda=1.6$} & Algorithm \ref{first-splitting} & $18.37/3921$ & $28.31/21132$ & $29.78/-$ & $29.78/-$   \\
  & Algorithm \ref{first-pre-splitting} & $26.89/502$ & $30.00/1521$ & $30.99/4223$ & $31.43/23334$   \\
  & Algorithm \ref{second-splitting} & $17.78/2791$ & $28.79/21704$ & $30.32/-$ & $30.32/-$  \\
    & Algorithm \ref{second-pre-splitting} & $27.00/471$ & $30.05/1476$ & $30.98/3751$ & $31.44/20318$   \\
    \hline
    \multirow{4}[1]{*}{$\lambda=1.8$} & Algorithm \ref{first-splitting} & $18.26/4006$ & $28.10/21545$ & $29.80/-$ & $29.80/-$   \\
  & Algorithm \ref{first-pre-splitting} & $26.74/504$ & $29.92/1518$ & $31.06/4278$ & $31.59/20884$   \\
  & Algorithm \ref{second-splitting} & $17.73/2920$ & $28.52/21850$ & $30.25/-$ & $30.25/-$  \\
    & Algorithm \ref{second-pre-splitting} & $26.80/478$ & $29.98/1490$ & $31.02/3889$ & $31.64/21790$   \\
    \hline
    \multirow{4}[1]{*}{$\lambda=2$} & Algorithm \ref{first-splitting} & $17.92/3942$ & $27.73/21777$ & $29.40/-$ & $29.40/-$   \\
  & Algorithm \ref{first-pre-splitting} & $26.30/506$ & $29.53/1529$ & $30.76/4410$ & $31.33/20260$   \\
  & Algorithm \ref{second-splitting} & $17.62/3073$ & $27.94/22010$ & $29.79/-$ & $29.79/-$  \\
    & Algorithm \ref{second-pre-splitting} & $26.32/487$ & $29.72/1552$ & $30.73/4134$ & $31.35/19732$   \\
    \hline
    \end{tabular}%
  \label{first-nonnegative}%
\end{table}%

\begin{table}[htbp]
\footnotesize
  \centering
  \caption{Comparison of the performance of Algorithms \ref{first-splitting}, \ref{second-splitting}, \ref{first-pre-splitting},
  and \ref{second-pre-splitting} in terms of SNR and iteration numbers with box constraints, i.e., $C= \{ x | 0 \leq
  x\leq
  1\}$}.
    \begin{tabular}{cccccc}
   \hline
\multirow{2}[1]{*}{Regularization} & \multirow{2}[1]{*}{Methods} &    $\epsilon = 10^{-3}$    &  $\epsilon = 10^{-4}$      &   $\epsilon = 10^{-5}$      &  $\epsilon = 10^{-6}$    \\
     Parameter      & & $SNR (dB)/k$  &  $SNR(dB)/k$ & $SNR(dB)/k$ & $SNR(dB)/k$  \\
    \hline
    \multirow{4}[1]{*}{$\lambda=0.6$} & Algorithm \ref{first-splitting} & $19.20/2899$ & $24.86/14335$ & $26.05/35501$ & $26.15/-$   \\
  & Algorithm \ref{first-pre-splitting} & $24.65/380$ & $26.65/1208$ & $27.23/5966$ & $27.31/22942$   \\
  & Algorithm \ref{second-splitting} & $19.08/655$ & $26.55/12137$ & $27.18/28035$ & $27.24/-$  \\
    & Algorithm \ref{second-pre-splitting} & $25.03/343$ & $26.78/1123$ & $27.25/4450$ & $27.30/16060$   \\
    \hline
\multirow{4}[1]{*}{$\lambda=0.8$} & Algorithm \ref{first-splitting} & $19.21/2969$ & $25.94/15457$ & $27.40/37330$ & $27.48/-$   \\
  & Algorithm \ref{first-pre-splitting} & $25.25/384$ & $27.83/1273$ & $28.65/5385$ & $28.77/22840$   \\
  & Algorithm \ref{second-splitting} & $19.28/799$ & $27.58/14095$ & $28.47/32327$ & $28.55/-$  \\
    & Algorithm \ref{second-pre-splitting} & $25.53/353$ & $28.08/1244$ & $28.71/14688$ & $28.79/18123$   \\
    \hline
\multirow{4}[1]{*}{$\lambda=1.2$} & Algorithm \ref{first-splitting} & $19.08/3187$ & $27.46/17257$ & $29.09/-$ & $29.09/-$   \\
  & Algorithm \ref{first-pre-splitting} & $25.90/396$ & $29.23/1342$ & $30.22/14588$ & $30.48/22166$   \\
  & Algorithm \ref{second-splitting} & $19.00/1132$ & $28.61/16171$ & $29.80/37411$ & $29.84/-$  \\
    & Algorithm \ref{second-pre-splitting} & $26.13/354$ & $29.30/1309$ & $30.22/3899$ & $30.50/19944$   \\
    \hline
\multirow{4}[1]{*}{$\lambda=1.6$} & Algorithm \ref{first-splitting} & $18.79/3346$ & $28.00/18658$ & $29.78/-$ & $29.78/-$   \\
  & Algorithm \ref{first-pre-splitting} & $26.11/413$ & $29.94/1422$ & $30.98/4097$ & $31.44/23096$   \\
  & Algorithm \ref{second-splitting} & $18.52/1478$ & $28.96/17626$ & $30.38/-$ & $30.38/-$  \\
    & Algorithm \ref{second-pre-splitting} & $26.41/378$ & $29.97/1378$ & $31.00/3790$ & $31.45/20929$   \\
    \hline
    \multirow{4}[1]{*}{$\lambda=1.8$} & Algorithm \ref{first-splitting} & $18.42/3366$ & $27.91/19126$ & $29.78/-$ & $29.78/-$   \\
  & Algorithm \ref{first-pre-splitting} & $25.89/415$ & $29.91/1462$ & $31.08/4373$ & $31.59/21341$   \\
  & Algorithm \ref{second-splitting} & $18.38/1645$ & $28.78/18234$ & $30.29/-$ & $30.29/-$  \\
    & Algorithm \ref{second-pre-splitting} & $26.20/392$ & $29.92/1419$ & $31.03/3922$ & $31.63/21140$   \\
    \hline
    \multirow{4}[1]{*}{$\lambda=2$} & Algorithm \ref{first-splitting} & $18.04/3432$ & $27.65/19592$ & $29.44/-$ & $29.44/-$   \\
  & Algorithm \ref{first-pre-splitting} & $25.53/422$ & $29.52/1488$ & $30.76/4380$ & $31.33/19802$   \\
  & Algorithm \ref{second-splitting} & $18.31/1858$ & $28.13/18920$ & $29.81/-$ & $29.81/-$  \\
    & Algorithm \ref{second-pre-splitting} & $25.89/408$ & $29.56/1464$ & $30.73/4070$ & $31.35/20100$   \\
    \hline
    \end{tabular}%
  \label{first-bounded}%
\end{table}%

The results in Tables \ref{first-nonnegative} and \ref{first-bounded}
indicate that preconditioned iterative Algorithm
\ref{first-pre-splitting} and Algorithm \ref{second-pre-splitting}
converge faster than iterative Algorithm \ref{first-splitting}
and Algorithm \ref{second-splitting}, respectively. The second class
of splitting primal-dual proximity algorithms (Algorithm
\ref{second-splitting} and Algorithm \ref{second-pre-splitting})
achieve higher SNR values than the first class (Algorithm
\ref{first-splitting} and Algorithm \ref{first-pre-splitting}),
respectively.

In addition, the results in Tables \ref{first-nonnegative} and
\ref{first-bounded} show that when the error tolerance decreases,
the SNR value increases accordingly; however, this requires a
larger number of iterations and is more time consuming.
The regularization parameter also has an impact on the
performance of these iterative algorithms. A large regularization
parameter means that the total variation term is strongly penalized.
We found the SNR value to increase as we increased the
regularization parameter; however, the SNR value was observed to
decrease when the regularization parameter exceeded the value of 2.

A comparison between Tables \ref{first-nonnegative} and
\ref{first-bounded} revealed that the SNR values of the reconstructed
images are very similar for the given regularization parameter level. The
reconstructed images are shown in Figure
\ref{experiment1-nonnegative-fig2} and Figure
\ref{experiment1-bounded-fig3}, where the regularization parameter
$\lambda = 1.8$ and the tolerance $\epsilon = 10^{-6}$.

\begin{figure} %[htbp]
\centering \setlength{\floatsep}{0pt}
\setlength{\abovecaptionskip}{-20pt}
\scalebox{1}{\includegraphics[width=1\textwidth]{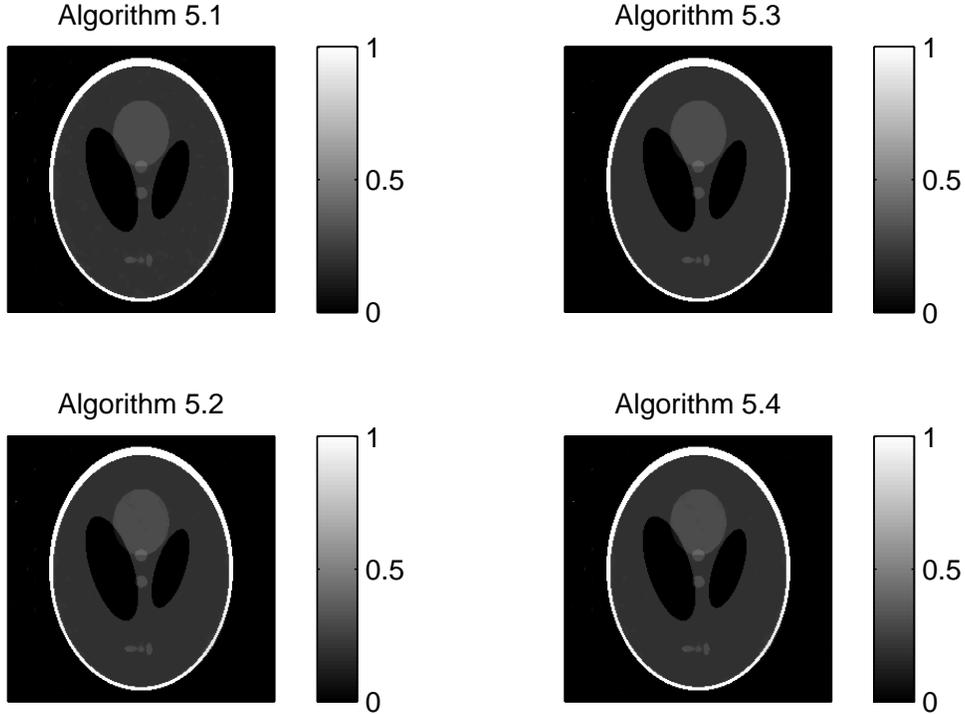}}
\caption[]{Reconstructed images obtained from Algorithms
\ref{first-splitting}, \ref{second-splitting}, \ref{first-pre-splitting},
and \ref{second-pre-splitting}
with non-negativity constraints, respectively.
}\label{experiment1-nonnegative-fig2}
\end{figure}

\begin{figure} %[htbp]
\centering \setlength{\floatsep}{0pt}
\setlength{\abovecaptionskip}{-20pt}
\scalebox{1}{\includegraphics[width=1\textwidth]{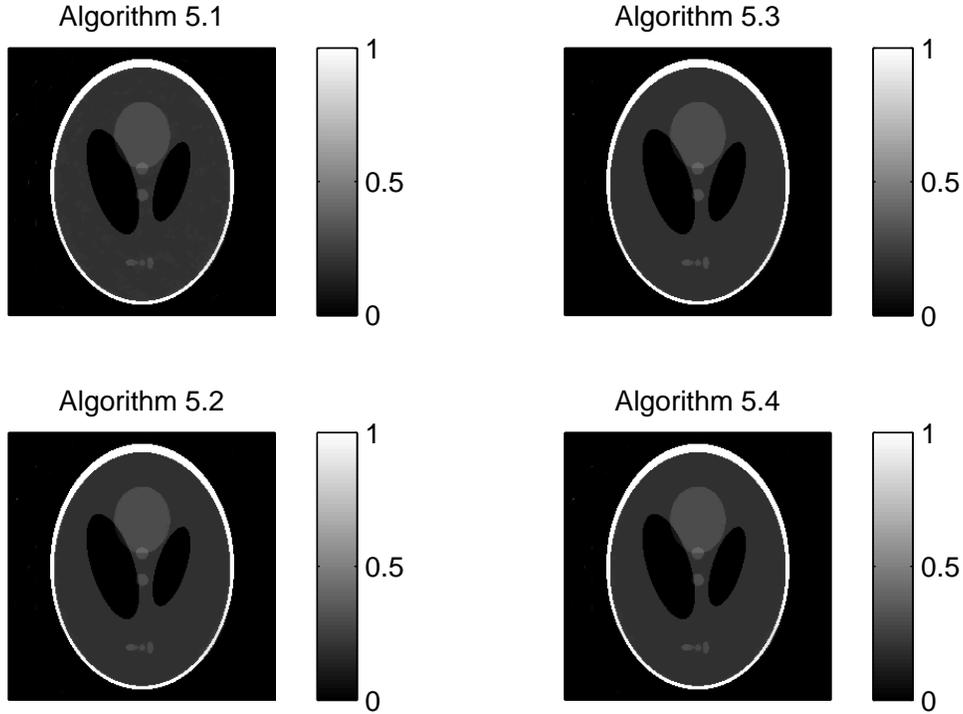}}
\caption[]{Reconstructed images obtained from Algorithms
\ref{first-splitting}, \ref{second-splitting},
\ref{first-pre-splitting}, and \ref{second-pre-splitting}
with box constraints, respectively.}\label{experiment1-bounded-fig3}
\end{figure}

\section{Conclusions}

In this paper, we proposed a splitting primal-dual proximity
algorithm to solve the general optimization problem
(\ref{problem1}). As its iterative parameters rely on
estimating some operator norm, this may affect its practical use.
Thus, we introduced a precondition technique to compute the iterative
parameters self-adaptively. Under some mild assumptions,
we proved the theoretical convergence of both iterative algorithms.
The methods proposed in this paper have been applied to the
constrained optimization model (\ref{constrained-l2-l1-tv}), which
has wide application in image restoration and image reconstruction
problems. We verified the numerical performance of these iterative
algorithms by applying them to CT image reconstruction problems. The
numerical results were very promising.

Although we have illustrated the use of our proposed methods in the
context of a CT image reconstruction problem, the proposed methods can
also be used to solve other application problems such as image deblurring
and denoising, and statistical learning problems.

\bigskip
\bigskip

\noindent \textbf{Acknowledgements}

This work was additionally supported by the National Natural Science
Foundations of China (11131006, 11201216, 11401293, 11461046), the
National Basic Research Program of China (2013CB329404), and the Natural
Science Foundations of Jiangxi Province (20151BAB211010,
20142BAB211016).

%\bibliographystyle{IEEEtran}
%\bibliography{klreference-en}

\end{document}